\documentclass[12pt,final]{amsart}
\usepackage{amsmath}
\usepackage{amsfonts}
\usepackage{amssymb}
\usepackage{graphicx} 

\usepackage[color,notref,notcite]{showkeys}  
\definecolor{labelkey}{rgb}{0.6,0,1}

\usepackage[normalem]{ulem}
\newcounter{corr}
\definecolor{violet}{rgb}{0.580,0.,0.827}
\newcommand{\corr}[3]{\typeout{Warning : a correction remains in page
\thepage}
				\stepcounter{corr}        
				{\color{blue}\ifmmode\text{\,\sout{\ensuremath{#1}}\,}\else\sout{#1}\fi}
       {\color{red}#2}
       {\color{violet} #3}}
\numberwithin{equation}{section}
 
 \newtheorem{thm}{Theorem}[section]
 \newtheorem{lem}[thm]{Lemma}
 \newtheorem{exam}[thm]{Example}
 \newtheorem{prop}[thm]{Proposition}
 \newtheorem{cor}[thm]{Corollary}
 \newtheorem{rem}[thm]{Remark}
 \newtheorem{defn}[thm]{Definition}

 \def\Id{\mathop{\rm Id}\nolimits}
 \def\di{\mathop{\rm div}\nolimits}
 \def\re{\mathop{\rm Re}\nolimits}
 \def\ran{\mathop{\rm Ran}\nolimits}
 
 \def\tr{\mathop{\rm tr}\nolimits}

 \def\dom{D}

 \newcommand{\Hm }{{H_{\mbox{\small{-}}}}}
 \newcommand{\Hp }{{H_{\mbox{\tiny{+}}}}}
 \newcommand{\Gm  }{G_{\mbox{\small{-}}}}
 \newcommand{\Gp  }{G_{+}}
 \newcommand{\xp }{x_+}
 \newcommand{\xm }{x_{\mbox{\small{-}}}}
 \newcommand{\hp }{h_+}
 \newcommand{\hm }{h_{\mbox{\small{-}}}}
\newcommand{\C}{{\mathbb C}}   
 \newcommand{\R}{{\mathbb R}}           
\newcommand{\K}{{\mathbb K}}        
\newcommand{\D}{{\mathbb D}}

\author{W. Arendt}
\address{Wolfgang Arendt, Institute of Applied Analysis, University of Ulm. Helmholtzstr. 18, D-89069 Ulm (Germany)} 
\email{wolfgang.arendt@uni-ulm.de}

\author{I. Chalendar}
\address{Isabelle Chalendar,  Universit\'e Gustave Eiffel, LAMA, (UMR 8050), UPEM, UPEC, CNRS, F-77454, Marne-la-Vallée (France)}
\email{isabelle.chalendar@univ-eiffel.fr}

\author{R. Eymard}
\address{Robert Eymard,  Universit\'e Gustave Eiffel, LAMA, (UMR 8050), UPEM, UPEC, CNRS, F-77454, Marne-la-Vallée (France)}
\email{robert.eymard@univ-eiffel.fr}
\title[Derivations and symmetric operators]{Extensions of derivations and symmetric operators}
\keywords{ }
\subjclass[2010]{65N30,47A07,47A52,46B20}
 
\begin{document}	

\begin{abstract}   
Given a densely defined skew-symmetric operators $A_0$ on a real or complex Hilbert space $V$, we parametrize all $m$-dissi\-pa\-tive extensions in terms of contractions $\Phi:\Hm \to \Hp  $, where 
$\Hm $ and $\Hp  $ are Hilbert spaces associated with a \emph{boundary quadruple}. Such an extension generates a unitary $C_0$-group if and only if $\Phi$ is a unitary operator. As corollary we obtain the parametrization
 of all selfadjoint  extensions of a symmetric operator by unitary operators from $\Hm $ to $\Hp  $. Our results extend the theory of boundary triples initiated by von Neumann and developed by V. I.  and M. L. Gorbachuk, J. Behrndt and M. Langer, S. A. Wegner and many others, in the sense that a boundary quadruple always exists (even if the defect indices are different in the symmetric case).   
 \end{abstract}	
\maketitle
\tableofcontents   
\section{Introduction}
A classical  subject in Functional Analysis and in Mathematical Phy\-sics is the description of all selfadjoint extensions of a symmetric operator.  It was J. von  
    Neumann who gave the first result in this direction (see Section~\ref{sec:5}). Generalizations of von Neumann's Theorem led to the theory of boundary triples which are described in detail in the monographs \cite{GG} by V.I. and M.L. Gorbachuk and by K. Schm\"udgen \cite{Sch12}. Such boundary triples exist whenever the given symmetric operator has at least one selfadjoint extension.  Multiplying by the complex number $i$, the problem can be reformulated as follows. Let $A_0$ be a densely defined skew-symmetric operator on a  Hilbert space $V$. Describe all  extensions $B$ of $A_0$ which generate a unitary $C_0$-group. 
    
   More recently,  Wegner \cite{Weg} used boundary triples to parametrize also all extensions $B$ of $A_0$ which generate a contractive $C_0$-semigroup (which does not necessarily consist of unitary operators).    However, it turns out  that for this task the notion of boundary triples is too narrow and does not cover all cases. 
   
   In the present article we introduce \emph{boundary quadruples}, a notion with weaker assumptions, which covers all cases. Moreover, we give quite short proofs which might make the new setting attractive. 
   
   Let us describe in more details some of the main results. Let $V$ be a Hilbert space over $\K=\R$ or $\C$, and let $A_0$ be a skew-symmetric operator with dense domain. Then $A_0$ is a restriction of $A:=-A_0^*$. 
   
   A \emph{boundary quadruple} $(\Hm,\Hp, \Gm  ,\Gp  )$ consists of pre-Hilbert spaces $\Hm,\Hp$ and surjective linear mappings $\Gm:\dom(A)\to \Hm$ and $\Gp:\dom(A)\to \Hp$ such that 
   \begin{equation}\label{eq:1.1}
   	\langle Au,v\rangle_V +\langle u,Av\rangle_V =\langle \Gp  u,\Gp  v\rangle_{\Hp} -\langle \Gm  u,\Gm  v\rangle_{\Hm}
   \end{equation}      
    for all $u,v\in D(A)$, with the additional condition
    \begin{equation}\label{eq:1.2}
    \ker \Gm   +  \ker \Gp  =D(A). 
    \end{equation}
    It is remarkable that the purely algebraic assumptions \eqref{eq:1.1} and \eqref{eq:1.2} imply that the mappings $\Gm  $ and $\Gp  $ are continuous and that the images $\Hm,\Hp$ are actually Hilbert spaces.  Our main result is the following. 
    \begin{thm}
    	Let $B$ be an extension of $A_0$. The following are equivalent:
    	\begin{itemize}
    	\item[(i)] $B$ generates a $C_0$-semigroup of contractions;
    	\item[(ii)] there exists a linear contraction $\Phi:\Hm   \to \Hp  $ such that 
    	\[  D(B)=\{ w\in D(A):\Phi \Gm  w=\Gp  w\}\mbox{ and }Bw=Aw,\,\, w\in D(B).\]
    	\end{itemize}	
    \end{thm}
    	Moreover, $B$ generates a unitary $C_0$-group if and only if $\Phi$ is unitary. In this case we refind the known extension results for symmetric operators in a  more general setting. 
     
    	The literature on boundary triples is very rich. We refer to the references and historical notes in the article \cite{Weg} by Wegner, the two monographs mentioned above, and also to the monograph of Behrndt, Hassi and de Snoo \cite{BHS20}, where extension results are part of an elaborate theory (cf. \cite[Corollary~2.1.4.5]{BHS20}, and we refer also  to the articles of Behrndt--Langer \cite{BL} and Behrndt--Schlosser \cite{BS}. In a previous article \cite{RDV1}, the present authors studied extensions of derivations in a quite different spirit, the main motivation being non-autonomous evolution equations. Even though some ideas used here have their origine in \cite{RDV1}, the present article is completely self-contained. It is organized as follows.      In Section~\ref{sec:2} we investigate dissipative operators including a simple proof of Phillips' Theorem. Boundary quadruples are introduced and investigated in Section~\ref{sec:3}. Extensions which generate a unitary group are characterized in Section~\ref{sec:4}. These results can be transferred to extension results for symmetric operators (Section~\ref{sec:5}). Examples are given in Section~\ref{sec:6}. Section~\ref{sec:6,5} is devoted to the wave equation.  Finally we compare boundary triples as they occur in the literature with  boundary quadruples in Section~\ref{sec:7}.

\section{Dissipative and skew-symmetric operators}\label{sec:2}
Let $V$ be a  Hilbert space over $\K=\R$ or $\C$. We first recall the definitions of several notions associated with dissipativity.  By an operator $A$  on $V$ we always understand  a linear mapping, defined on  a subspace $\dom(A)$ of $V$, its domain, which takes values in $V$.  
\begin{defn}\label{def:various}
	\begin{itemize}
		\item[a)] An operator $A$ with domain  $\dom(A)\subset V$ is \emph{closed} if its graph ${\mathcal G}(A):=\{(x,Ax):x\in\dom(A)\}$ is closed in $V\times V$. 
		\item[b)] An operator $B$ is an \emph{extension} of $A$, in symbols $A\subset B$, if $\dom(A)\subset \dom(B)$ and 
		$Bx=Ax$ for $x\in \dom(A)$. 
		\item[c)] An operator $A$ is called  \emph{closable} if there exists a closed operator  $\overline{A}$ such that ${\mathcal G}(\overline{A})=\overline{{\mathcal G}(A)}$. 
		Then $\overline{A} $ is the smallest closed extension  of $A$. It is called the \emph{closure} of $A$. 
		\item[d)] An operator $A$ on $V$ is called \emph{dissipative} if 
		\[ \re  \langle Ax,x\rangle_V \leq 0 \mbox{ for all }x\in\dom (A). \]
		\item[e)] An operator $A$ on $V$ is called \emph{ maximal dissipative} if $A\subset B$ with $B$ a  dissipative operator  implies that $A=B$.
		\item[f)]  An operator $A$ on $V$ is called \emph{$m$-dissipative} if $A$ is dissipative and $(\Id -A)\dom(A)=V$.  
	\end{itemize}  
\end{defn}  
We now collect well-known properties to see the link between all the notions defined above.
\begin{prop}\label{prop:standard} 
	Let $A$ be an operator on $V$. 
\begin{itemize}
	\item[1)] The operator $A$ on $V$ is dissipative if and only if 
	\begin{equation}\label{eq:2.1n}
		\|x-tAx\|\geq \|x\|\mbox{ for all }x\in\dom (A),\,\, t>0.
	\end{equation} 
\item[2)] If $A$ is $m$-dissipative, then $\dom(A)$ is dense in $V$. 
\item[3)] If $A$ is $m$-dissipative, then $A$ is maximal dissipative. 
\item[4)] If $A$ is dissipative and $\dom(A)$ is dense in $V$, then $A$ is closable and $\overline{A}$ is dissipative.
\item[5)] If $A$ is dissipative and closed, then, for all $t>0$, $(\Id-tA)\dom(A)$ is closed.    
\end{itemize}
\end{prop}
\begin{proof}  
1) follows from the equivalence 
\[ \|(Id -tA)x\|_V^2\geq \|x\|_V^2\Longleftrightarrow t\|Ax\|_V^2-2\re \langle Ax,x\rangle_V\geq 0,\]
Then, letting $t\to 0$, we prove the well-known characterization of dissipativity.\\
2)  is proved in \cite[Proposition 3.3.8]{ABHN11}. \\
3) Let $A\subset B$ where $B$ is dissipative. Let $x\in \dom(B)$. The surjectivity of $\Id -A$ implies that there exists 
$y\in\dom(A)$ such that $x-Bx=y-Ay=y-By$.  Therefore $B(y-x)=y-x$. By 1) applied with $t=1$, we get $x=y$, and then $A=B$. \\   
4) follows from  \cite[Lemma~3.4.4]{ABHN11} and  1).\\ 
5) Suppose that $(\Id-tA)(x_n)\to y$ for $(x_n)_n\subset \dom(A)$. Therefore   $((\Id-tA)(x_n))_n$ is a Cauchy sequence and by 1) $(x_n)_n$ is also a Cauchy sequence, which converges to say $x\in V$. It follows that $Ax_n\to \frac{x-y}{t}$. Since the graph of $A$ is closed, $Ax=\frac{x-y}{t}$ and thus $y=(\Id- tA)x$.    
\end{proof}

Phillips \cite[Corollary of Theorem 1.1.1]{Ph59} obtained the equivalence between $m$-dissipativity and maximal dissipativity on a Hilbert space  using the Cayley transform. We give a simple direct proof (see also \cite[Theorem 3.1]{Weg} for still another argument). 

\begin{thm}[Phillips]\label{th:2.1n}  
	Let $V$ be a Hilbert space and $A$ an operator on $V$. Then $A$ is $m$-dissipative if and only if $A$ is maximal dissipative and $\dom(A)$ is dense. 
 \end{thm}
\begin{proof}
	One implication follows from 3) in Proposition~\ref{prop:standard}.  \\
	Conversely, if $A$ is dissipative with dense domain, then, by  4) in Proposition~\ref{prop:standard}, $A$ is closable and $\overline{A}$ is dissipative. Since $A$ is maximal dissipative, $A=\overline{A}$. By 5) in  Proposition~\ref{prop:standard} with $t=1$, we get $(\Id-A)\dom(A)$ is closed in $V$.  
	Assume that $R:=\ran (\Id -A)\neq V$.   Let us prove that $\dom (A)\cap R^\perp=\{0\}$. Let $x\in \dom(A)\cap R^\perp$. Then 
	\[ \langle x-Ax,x\rangle_V=0.\]
	Hence    $\|x\|_V^2=\re \langle Ax,x\rangle_V\leq 0$ and so $x=0$.  Define $\widehat{A}$ on $V$ by 
	\[ \dom(\widehat{A}):=\dom(A)\oplus R^\perp, \,\, \widehat{A}(x+u)=Ax-u\mbox{ where } x\in\dom(A),u\in R^\perp.\]
	Then $A\subset \widehat{A}$ and 
	\begin{eqnarray*}
		\re \langle \widehat{A}(x+u),x+u\rangle_V & = & \re \langle Ax,x\rangle_V +\re \langle Ax,u\rangle_V  -\re \langle u,x\rangle_V -\|u\|_V^2\\
		& \leq & \re (\langle Ax,u\rangle_V -\langle u,x\rangle_V)\\
		& = &  \re (-\langle x-Ax,u\rangle_V) +\re (\langle x,u\rangle_V-\langle u,x\rangle_V)\\
		& = & \re ( \langle x,u\rangle_V -\overline{\langle x,u\rangle_V})\\
		& = & 0.
	\end{eqnarray*}	
	Thus $\widehat{A}$ is dissipative. Hence $\dom (A)=\dom(\widehat{A})$, i.e. $R^\perp =\{0\}$. We have shown that $R=V$; i.e. $A$ is $m$-dissipative.   
\end{proof}

By Zorn's lemma, each densely defined dissipative operator has a maximal dissipative extension. 

By the Lumer--Phillips Theorem \cite[Theorem 3.4.5 and Proposition 3.3.8]{ABHN11}, an operator $A$ on $V$ generates a $C_0$-semigroup of contractions if and only if $A$ is $m$-dissipative. 

Now we introduce the basic objects of this article, a skew-symmetric operator which can be characterized as follows.  We continue to consider the real and complex case simultaneously. 
\begin{prop}\label{prop:2.2}
Let $A$ be an operator on $V$. The following assertions are equivalent:
\begin{itemize}
	\item[(i)] $\re \langle Au,u\rangle_V=0$ for all $u\in D(A)$;
	\item[(ii)] $\pm A$ is dissipative;
	\item[(iii)] $A$ is \emph{skew-symmetric}, i.e. 
	\[   \langle Au,v\rangle_V +\langle u,Av\rangle_V=0\]
	for all $u,v\in \dom (A)$.   
\end{itemize}
\end{prop} 
\begin{proof}
$(i)\Longleftrightarrow (ii)$ is obvious.\\
$(iii)\Rightarrow (i)$ Take $u=v$. \\
$(i)\Rightarrow (iii)$ Note that 
\begin{eqnarray*}
0 & = & \re \langle A(u+v),u+v\rangle_V\\
 & = & \re \left( \langle Au,v\rangle_V + 	\langle Av,u\rangle_V\right)\\
  & = & \re \left( \langle Au,v\rangle_V + 	\langle u,Av\rangle_V\right).
\end{eqnarray*}	 
This proves $(iii)$ in the real case. If $\K=\C$, then replacing $u$ by $\lambda u$ we see that 
\[ \re \lambda \left( \langle Au,v\rangle_V + 	\langle u,Av\rangle_V\right)\]
for all $\lambda\in\C$ and hence $\langle Au,v\rangle_V +\langle u,Av\rangle_V=0$
for all $u,v\in \dom (A)$.   
\end{proof}
If $B$ is a densely defined operator on $V$, the adjoint $B^*$ of $B$ is defined as follows.  For $v,f\in V$, 
\[v\in \dom( B^*) \mbox{ and }B^*v=f \Longleftrightarrow \langle Bu,v\rangle_V =\langle u,f\rangle_V,\,\,u\in \dom(B).\]
Using this definition we see that a densely defined operator $A_0$ on $V$ is skew-symmetric if and only if $A_0\subset (-A_0)^* $. 

Our aim is to describe all $m$-dissipative extensions of such an operator $A_0$. It turns out that they all are restrictions of $(-A_0)^*$. A proof using Cayley transform is given in \cite[Chapter 3, Theorem 1.3 p. 150]{GG} and \cite[Proposition 2.8]{Weg}. We give a much shorter direct proof. 
\begin{thm}\label{th:2.3}
Let $A_0$ be a densely defined skew-symmetric operator  and $B$ a dissipative operator such that $A_0\subset B$. Then   $B\subset (-A_0)^*$. 
\end{thm}  
\begin{proof}
Let $y\in\dom (B)$. Let $x\in \dom(A_0)$. Then 
\[ \re \langle Bx,x\rangle _V =\re \langle A_0 x,x\rangle_V=0 . \]
Hence, for $t>0$, 
\begin{eqnarray*}
0 & \leq & \re \langle B(x+ty),x+ty\rangle_V	\\
 & = & t\re \langle Bx,y\rangle_V + t\re \langle By,x\rangle_V + t^2 \re \langle By,y\rangle_V.	
\end{eqnarray*}
Dividing by $t$ and letting $t\to 0$, we get
\[   0\leq  \re \langle A_0 x,y\rangle_V +\re \langle By,x\rangle_V.\]
Since the above inequality holds for $\pm x$, it follows that
 \[   0= \re \langle A_0 x,y\rangle_V +\re \langle By,x\rangle_V.\]
If $\K=\R$, since $x\in \dom(A_0)$ is arbitrary, this implies that $y\in \dom(A_0^*)$ and $By=(-A_0)^* y$. \\
If $\K=\C$, replacing $t$ by $\lambda t$ with $\lambda\in\C$, the above argument shows that 
\begin{equation}\label{eq:2.1}
\re \left( \overline{\lambda}\langle A_0 x,y\rangle_V +\lambda \langle By,x\rangle_V \right) = 0	
\end{equation}  
for all $\lambda\in \C$ and $x\in\dom(A_0)$. 
Choosing $\lambda =1$ and then $\lambda=i$, we get 
\[  \langle A_0 x,y\rangle_V +  \langle x,By\rangle_V=0,\]
and the  conclusion follows as  in the case $\K=\R$.  
\end{proof}

\section{Boundary quadruples and $m$-dissipative restrictions}\label{sec:3}
In this section we introduce and study the basic notion of this article.  
Let $V$ be a  Hilbert space over $\K=\R$ or $\C$ and let $A_0$ be a densely defined skew-symmetric operator on $V$. Let $A=(-A_0)^*$. Then, as a consequence of the definition of the adjoint 
\begin{equation}\label{eq:3.1}
	A_0\subset A.
\end{equation}	
\begin{defn}\label{def:3.1}
A \emph{boundary quadruple} $(\Hm,\Hp,\Gm  ,\Gp  )$ for $A_0$ consists of  pre-Hilbert spaces $\Hm,\Hp$ and linear maps $\Gm:\dom(A)\to \Hm$, $\Gp:  \dom(A)\to \Hp$ satisfying  
\begin{equation}\label{eq:3.2} 
	\langle Au,w\rangle_V +\langle u,Aw\rangle_V=\langle \Gp   u,\Gp  w\rangle_{\Hp} -\langle \Gm   u,\Gm   w\rangle_{\Hm}
\end{equation}
for all $u,w\in\dom(A)$, 
\begin{equation}\label{eq:3.3}
	\ker \Gp   +\ker \Gm  =\dom(A),
\end{equation} 
and 
\begin{equation}\label{eq:3.4n}
	\Hm=\Gm \dom(A),\,\,    \Hp=\Gp\dom(A).
\end{equation} 
\end{defn}
We will see below that these assumptions imply that $\Hm$ and $\Hp$ are actually complete and that $\Gm,\Gp$ are continuous (with respect to the graph norm on $\dom(A)$).
 
Boundary quadruples do always exist. Below are two possible constructions which always work. However, for concrete examples, other choices might be convenient. We let   
\begin{equation}\label{eq:3.4}
b(u,v)=\langle Au,v\rangle_V + \langle u,Av\rangle_V,\,\, u,v\in \dom(A). 	
\end{equation} 
Then $b:\dom(A)\times \dom(A)\to \K$ is a symmetric sesquilinear form. We call it the \emph{boundary form} associated with $A_0$.

We first show that Condition \eqref{eq:3.3} implies the following \emph{interpolation property}. 
\begin{lem}\label{lem:3.2}
	Let $\xm    \in \Gm  \dom(A)$,  $\xp    \in \Gp   \dom(A)$. Then there exists $w\in\dom(A)$ such that $\Gm   w=\xm    $ and $\Gp  w=\xp    $.  
\end{lem}
   \begin{proof}
   	There exist $w_1,w_2\in\dom(A)$ such that $\xm    =\Gm  w_1$, $\xp    =\Gp   w_2$. By \eqref{eq:3.3}, $w_1=w_{1-} + w_{1+}$ and  $w_2=w_{2-} + w_{2+}$ with $w_{1-},w_{2-}\in \ker \Gm  $,  $w_{1+},w_{2+}\in \ker \Gp  $. Let $w=w_{1+} + w_{2-}$. Then 
   	\[  \Gm   w=\Gm  w_{1+}=\Gm  (w_{1+} + w_{1-})=\Gm  w_1=\xm    .\] 
   	Similarly
   	\[  \Gp   w=\Gp  (w_{1+}+w_{2-})=\Gp   w_{2-} = \Gp  (w_{2-} + w_{2+})=\Gp  w_2=\xp    .\] 
   \end{proof}
\begin{rem}\label{rem:3.3}
	The interpolation property of Lemma~\ref{lem:3.2} is equivalent to \eqref{eq:3.3}. Indeed, given  $w\in \dom(A)$ and assuming the interpolation property, we find $w_1\in \dom(A)$ such that $\Gp   w_1=\Gp   w$ and $\Gm  w_1=0$. Thus $w_1\in \ker \Gm  $ and $w-w_1\in \ker \Gp  $, $w=(w-w_1)+w_1\in \ker \Gp  +\ker \Gm  $.   
\end{rem}
We consider $\dom(A)$ with the graph norm
\[   \|u\|_{\dom(A)}^2:=\|u\|_V^2 +\|Au\|_V^2.\]
Then $\dom(A)$ is a Hilbert space with $\|\,\|_{\dom(A)}$ as corresponding norm. We first prove that $\Gm  $ and $\Gp  $ are automatically continuous for this norm. 
\begin{lem}\label{lem:3.4}
The operators $\Gm:\dom(A)\to \Hm$,  $\Gp  :\dom(A)\to \Hp$ are continuous. 	
\end{lem}	 
\begin{proof}
	Consider the mapping $G:\dom(A)\to \overline{\Hm}\times \overline{\Hp}$ given by 
	\[ G(w)=(\Gm  (w),\Gp  (w))\in \overline{\Hm}\times \overline{\Hp} ,\]
	where $\overline{\Hm}$, $\overline{\Hp}$ are the completions of $\Hm$ and $\Hp$ respectively.
	It suffices to show that $G$ has a closed graph. For that let $w_n\to w$ in $\dom(A)$ such that $G(w_n)\to (\xm    ,\xp    )\in \overline{\Hm}\times \overline{\Hp}$. Note that  $b$ is symmetric and hence automatically continuous as a consequence of the Closed Graph Theorem.   Thus 
	\begin{eqnarray*}
		\langle \xp    ,\Gp  v\rangle_{\Hp}-\langle \xm    ,\Gm  v\rangle_{\Hm} & = & \lim_{n\to\infty} \left(\langle \Gp  w_n,\Gp  v\rangle_{\Hp}-\langle \Gm  w_n,\Gm  v\rangle_{\Hm}	\right) \\
		 & = & \lim_{n\to\infty}b(w_n,v)\\
		  & = & b(w,v)\\
		  & = & \langle \Gp  w,\Gp  v\rangle_{\Hp} -\langle \Gm   w,\Gm  v\rangle_{\Hm}
		\end{eqnarray*}
	for all $v\in \dom(A)$. Thus 
	\[   \langle \xp    -\Gp  w,\Gp  v\rangle_{\Hp} +\langle \Gm  w-\xm    ,\Gm   v\rangle_{\Hm}=0\]
	for all $v\in\dom(A)$. Let $z_{\mbox{\tiny{-}}}\in \Gm  \dom(A)$, $z_+\in \Gp  \dom(A)$. By the interpolation property of Lemma~\ref{lem:3.2}, there exists $v\in\dom(A)$ such that $z_{\mbox{\tiny{-}}}=\Gm  v$, $z_+=\Gp   v$. Thus 
	\[  \langle \xp     -\Gp   w,z_+\rangle_{\Hp} + \langle \Gm   w-\xm    ,z_{\mbox{\tiny{-}}}\rangle_{\Hm}=0.\] 
	Passing to limits this remains true for all $z_{\mbox{\tiny{-}}}\in \overline{\Hm}$, $z_+\in \overline{\Hp}$.
	 
	Choosing $z_{+}=\xp    -\Gp  w$  and $z_{\mbox{\tiny{-}}}=\Gm  w-\xm$, it follows that 
	\[    \|\xp    -\Gp  w\|_{\Hp}^2 +\|\Gm  w-\xm    \|_{\Hm}^2=0.\]
	Hence $\xp    =\Gp  w$, $\xm    =\Gm  w$.   
	
\end{proof}
Next we show that $\Hm  $ and $\Hp  $ are complete. We could deduce it from \cite[Proposition~4.10]{RDV1} but we prefer to give a direct proof. 
\begin{prop}\label{prop:3.5}
The spaces $\Gm  \dom(A)=:\Hm$  and $\Gp  \dom(A)=:\Hp$ are complete.   
\end{prop}
\begin{proof}
a) We show that $\Gm  \dom(A)$ is closed in the completion $\overline{\Hm}$ of $\Hm$. Let $x_0\in\overline{\Gm  \dom(A)}$. Define the continuous linear form $F:\dom(A)\to\K$ by 
\[  F(u)=\langle \Gm  u,x_0\rangle_{\Hm}.\]
Then there exist 
$w_1,w_2\in V$
such that 
\[   F(u)=\langle u,w_1\rangle_V +\langle Au,w_2\rangle_V\]
for all $u\in\dom(A)$ (since each continuous linear from on $\dom(A)$ is of this form). In particular, 
\[   0=\langle u,w_1\rangle_V +\langle Au,w_2\rangle_V \]
for all $u\in \dom(A_0)$. Thus $w_1=Aw_2$. Consequently, 
\begin{eqnarray*}
\langle \Gm   u,x_0\rangle_{\Hm} & = & F(u)\\
 & = & \langle u,Aw_2\rangle_V +\langle Au,w_2\rangle_V\\
  & = & \langle \Gp  u,\Gp  w_2\rangle_{\Hp} -\langle \Gm  u,\Gm   w_2\rangle_{\Hm} 	
\end{eqnarray*}	     
for all $u\in\dom(A)$. In particular, 
\[   \langle \Gm  u, x_0 +\Gm   w_2\rangle_{\Hm} =0\]
for all $u\in \ker \Gp  $. Since, by \eqref{eq:3.3}, $\Gm  \ker \Gp  =\Gm  \dom(A)$, it follows that 
\[  \langle y,x_0 + \Gm   w_2\rangle_{\Hm}=0\]
for all $y\in \Gm  \dom(A)=\Hm$. Since $x_0+\Gm  w_2\in \overline{\Hm}$, it follows that $x_0 +\Gm  w_2=0$. Hence $x_0=-\Gm   w_2\in \Gm  \dom(A)=\Hm$. \\
b) Note that $(\Hp,\Hm,\Gp  ,\Gm  )$ is a boundary quadruple for $-A_0$.    Thus it follows from a) that $\Hp$ is complete. 
\end{proof}

Now we give two examples of  boundary quadruples which always exist.  
\begin{exam}[First example]\label{ex:3.6}
Consider $W:=\dom(A)$ with the graph norm, as before. Since $b$ is symmetric there exists a selfadjoint operator $B\in{\mathcal L}(W)$ such that $b(u,v)=\langle Bu,v\rangle_V$ for all $u,v\in V$. By the spectral theorem \cite[VIII.3]{RS80}, up to a unitary equivalence, $W=L^2(\Omega,\mu)$ and $Bf=mf$ for all $f\in W$ and some real-valued $m\in L^\infty(\Omega,\mu)$ where $(\Omega,\Sigma,\mu)$ is a measure space. Define $\Gp  f=\sqrt{m^+}f$ and $\Gm  f=\sqrt{m^-} f$ for all $f\in L^2(\Omega,\mu)$, where $m^+=\max(m,0)$ and $m^-=-\min(m,0)$. Then $(\Gm L^2(\Omega,\mu), \Gp L^2(\Omega,\mu),\Gm  ,\Gp  )$ is a  boundary quadruple.     
\end{exam}
\begin{exam}[Second example]
 The second construction depends on the following decomposition
\begin{equation}\label{eq:3.5} 
   \dom(A)=\dom(\overline{A_0})\oplus \ker (\Id -A)\oplus \ker (\Id+A),
\end{equation}
where $\overline{A_0}$ is the closure of $A_0$ (which is again skew-symmetric). A proof can be found in 
\cite[Lemma 2.5]{Weg} or \cite[Chapter 3, Theorem 1.1, p. 148]{GG} or \cite[X.1 Lemma]{RS75}. We give a proof to be as self-contained as possible, and we treat again the real and complex case simultaneously. First we note the decomposition
  \begin{equation}\label{eq:3.6} 
   V=\ran (\Id -\overline{A_0})\oplus \ker (\Id +A)
  \end{equation}    
which is a Hilbert direct sum since $\overline{\ran}(\Id -A_0)=\ran (\Id-\overline{A_0})$ (a consequence of dissipativity) and 
\[ \ran (\Id-A_0)^\perp=\ker (\Id-A_0^*)=\ker(\Id +A).\]
\begin{proof}[Proof]  of \eqref{eq:3.5}]
a) \emph{Linear independence}. Let $0=x_0 +\xp    +\xm    $ with 
\[x_0\in \dom(\overline{A_0}), \,\,\xp    ,\xm    \in\dom(A),\,\, A\xp    =\xp    ,\,\, A\xm    =-\xm     .\]
Then $0=(\Id -\overline{A_0})x_0 +2\xm    $. Since $2\xm    \in \ker (\Id+A)$, it follows from \eqref{eq:3.6} that $2\xm    =0$ and $(\Id-A_0)x_0=0$. From dissipativity we obtain now 
\[  0=\re \langle (\Id-A_0)x_0,x_0\rangle_V=\|x_0\|_V^2-\re \langle A_0x_0,x_0\rangle_V\geq \|x_0\|_V^2,\]
hence $x_0=0$. Thus also $\xp    =0$. \\
b) Let $x\in \dom(A)$. By \eqref{eq:3.6} there exist $x_0\in \dom(\overline{A_0})$ and $y_0\in \ker (\Id + A)$ such that 
  \[  (\Id - A)x=(\Id -\overline{A_0})x_0 +y_0=(\Id -A)(x_0+\frac{y_0}{2}).\]
  Thus $x-x_0-\frac{y_0}{2}   \in \ker (\Id -A)$ and 
  \[  x=x_0 +(x-x_0-\frac{y_0}{2}   ) +\frac{y_0}{2}    \in \dom(\overline{A_0}) +\ker (\Id -A) + \ker(\Id +A).\]	
  \end{proof}
\end{exam}
\begin{prop}\label{prop:3.7}
For $w\in\dom(A)$, let $w=w_0+w_+ +w_{\mbox{\tiny{-}}}$ be the decomposition according to \eqref{eq:3.5}, where $w_0\in\dom(\overline{A_0}),\,\, w_+,w_{\mbox{\tiny{-}}}\in\dom(A)$, $Aw_+=w_+, \,\,-A w_{\mbox{\tiny{-}}}=w_{\mbox{\tiny{-}}}$. Then $\Gp  w:=w_+$, $\Gm  w:=w_{\mbox{\tiny{-}}}$ defines a boundary quadruple $(\Hm,\Hp,\Gm  ,\Gp  )$, with $\Hm=\ker(\Id +A)$ and $\Hp=\ker(\Id -A)$.
\end{prop}
\begin{proof}
a) Let $u,v\in\dom(A)$. Then 
\[   b(u,v)=2\langle u_+,v_+\rangle_V-2\langle u_{\mbox{\tiny{-}}},v_{\mbox{\tiny{-}}}\rangle_V \]
with the notation defined above. In fact, using that, 
\[   \langle \overline{A_0}u_0,w\rangle_V=\langle -u_0,Aw\rangle_V\]
for all $w\in\dom(A)$, this is straightforward.\\
b) $\ker \Gp   +\ker \Gm  =\{   w\in\dom(A):w_+=0\}+\{w\in\dom(A):w_{\mbox{\tiny{-}}}=0\}=\dom(A)$. Thus \eqref{eq:3.2} and \eqref{eq:3.3} are satisfied.  
\end{proof}
Let $(\Hm,\Hp,\Gm  ,\Gp  )$ be a  boundary quadruple, which is fixed for the remainder of this section. We will decribe all $m$-dissipative extensions of $A_0$ in terms of this quadruple.  

Recall that $A_0$ is closable and that $\overline{A_0}$ is skew-symmetric again. The domain of $\overline{A_0}$ may be described by the  boundary quadruple as follows. 
\begin{prop}\label{prop:3.8}
One has $\dom (\overline{A_0})=\ker \Gp  \cap \ker \Gm  $ and $\overline{A_0}w=Aw$ for all $w\in \dom(\overline{A_0})$. 
\end{prop}
\begin{proof}
$"\subset"$ Let $u\in\dom(A_0)$. Then for all $v\in\dom(A)$, 
\[0=\langle \Gp u,\Gp v\rangle_{\Hp}-\langle \Gm u,\Gm v\rangle_{\Hm} .\]
By Lemma~\ref{lem:3.2} there exists $v\in\dom(A)$ such that $\Gp u=\Gp v$ and $\Gm v=-\Gm u$. Thus 
\[   0=\|\Gp u\|_{\Hp}^2 +\|\Gm u\|_{\Hm}^2.\]
Hence $u\in\ker \Gp\cap \ker \Gm$. 
Since $\Gp  $ and $\Gm  $ are continuous for the graph norm on $\dom(A)$, it follows that 
\[   \dom(\overline{A_0})\subset \ker \Gp  \cap \ker \Gm   .\] 
$"\supset"$ Let $w\in \ker \Gm  \cap \ker \Gp  $. Then 
\[ \langle Aw,u\rangle_V+\langle w,Au\rangle_V=\langle \Gp   w,\Gp  u\rangle_{\Hp} -\langle \Gm   w,\Gm  u\rangle_{\Hm} =0  \]
for all $u\in\dom(A)$. Hence 
\[  \langle w,A_0^*u\rangle_V=-\langle w,Au\rangle_V=\langle Aw,u\rangle_V \]
for all $u\in \dom(A_0^*)$. By \cite[Proposition B10, p. 472]{ABHN11} this implies that $w\in \dom(\overline{A_0}) $ and $\overline{A_0}w=Aw$. 
\end{proof}
Now we can describe all $m$-dissipative extensions of $A_0$. Recall that $\Hp  :=\ran \Gp  $ and $\Hm =\ran \Gm   $ are complete.
 If $\Phi\in{\mathcal L}(\Hm ,\Hp  )$ is a contraction, then we define the operator $A_\Phi$ on $V$ by 
\[ \dom (A_\Phi):=\{ w\in \dom(A):\Phi \Gm w = \Gp   w \}\mbox{ and }A_\Phi w:=Aw.\]
Then clearly $A_0\subset A_\Phi\subset A$. 
\begin{thm}\label{th:3.9}
Let $B$ be an operator on $V$ such that $A_0\subset B$. The following assertions are equivalent:
\begin{itemize}
	\item[(i)] $B$ is $m$-dissipative; 
	\item[(ii)] there exists a linear contraction $\Phi:\Hm \to \Hp  $ such that $B=A_\Phi$.  
\end{itemize}
\end{thm}   
 To say that $\Phi:\Hm\to\Hp$ is a  \emph{contraction}   means, by definition, that
 $\|\Phi x\|_{\Hp}\leq \|x\|_{\Hm}$ for all $x\in \Hm $. Since for $w\in \dom(A_\Phi)$, one has 
 \begin{eqnarray*}
 2\re \langle A_\Phi w,w\rangle_V	& = & b(w,w)\\
  & = & \|\Gp  w\|_{\Hp}^2-\|\Gm  w\|_{\Hm}^2\\
   & = & \|\Phi \Gm  w\|_{\Hp}^2-\|\Gm  w\|_{\Hm}^2\leq 0,
 \end{eqnarray*}	 
each operator $A_\Phi$ is dissipative. Before proving Theorem~\ref{th:3.9} we show that $A_\Phi$ determines $\Phi$. More precisely the following holds. 
\begin{prop}\label{prop:3.10}
Let $\Phi_1,\Phi_2:\Hm \to \Hp$ be two contractions. If $A_{\Phi_1}\subset A_{\Phi_2}$, then $\Phi_1=\Phi_2$. 
\end{prop}
\begin{proof}
	Let $x\in \Hm $. By Lemma~\ref{lem:3.2} there exists $w\in \dom(A)$ such that $\Gm  w=x$, $\Gp   w=\Phi_1(x)$. Thus $w\in \dom(A_{\Phi_1})$. It follows from the hypothesis that $w\in\dom(A_{\Phi_2})$, i.e. 
	\[  \Phi_2(x)=\Phi_2(\Gm  w)=\Gp  (w)=\Phi_1(x).  \] 
\end{proof}
\begin{proof}[Proof of Theorem~\ref{th:3.9}]
Let $B$ be a dissipative extension of $A_0$. Then by Theorem~\ref{th:2.3}, $B\subset A$. Since 
\[ 0\geq 2\re \langle Bu,u\rangle_V=b(u,u)=\|\Gp  u\|_{\Hp}^2-\|\Gm  u\|_{\Hm}^2, \]
one has $\|\Gp  u\|_{\Hp}^2\leq \|\Gm  u\|_{\Hm}^2$ for all $u\in\dom(B)$. Thus $\Phi \Gm  u:=\Gp  u$ defines a contraction from $\Gm  \dom (B)$ to $\Hp  $. Then $\phi$ has a contractive extension from $\Hm \to \Hp  $ which we still denote by $\Phi$ (extend first $\Phi$ to the closure of $\Gm  \dom(B)$ by continuity and then by $0$ on $(\Gm  \dom(B))^\perp$). From our definition it follows that $B\subset A_\Phi$. Proposition~\ref{prop:3.10}  now shows that the maximal dissipative operators extending $A_0$ are exactly the operators $A_\Phi$ where $\Phi:\Hm \to \Hp  $ is a linear contraction. Now the claim follows  from Phillips' Theorem (Proposition~\ref{prop:2.2}). 	
\end{proof}	
\begin{cor}\label{cor:3.11}
The mapping $\Phi\mapsto A_\Phi$ is a bijection between the set of all contractions $\Phi:\Hm \to \Hp  $ and the set of all $m$-dissipative extensions of $A_0$.   
\end{cor}
\begin{proof}
This follows from  Theorem~\ref{th:3.9} and Proposition~\ref{prop:3.10}. 
\end{proof}
Let us consider some special cases. 
\begin{prop}\label{prop:3.12}
	The following assertions are equivalent:
	\begin{itemize}
		\item[(i)] $\langle Au,v\rangle_V +\langle u,Av\rangle_V=0$  for all $u,v\in\dom(A)$; 
		\item[(ii)] $\Gm  =\Gp  =0$;
		\item[(iii)] $\overline{A_0}=A$;
		\item[(iv)] $-A^*=A$. 
	\end{itemize}
\end{prop}
\begin{proof}
$(ii)\Rightarrow (i)$ This follows from \eqref{eq:3.2}. \\
$(i)\Rightarrow (iv)$ Property $(i)$ implies that $-A\subset A^*$.  But $A^*=-A_0^{**}=-\overline{A_0}\subset -A$. This proves $(iv)$. \\
$(iv)\Rightarrow (iii)$ One always has $\overline{A_0}=A_0^{**}$ and $(\overline{A_0})^*=A_0^*$. Thus $(iv)$ implies that $\overline{A_0}=A_0^{**}=-A^*=A$. \\
$(iii)\Rightarrow (ii)$ This follows from Proposition~\ref{prop:3.8}.  
\end{proof}
Proposition~\ref{prop:3.12} describes when both operators $\Gp  ,\Gm  $ are zero; it turns out that this is independent of the  boundary quadruple. Next we consider the case when $\Gp  =0$.  
\begin{prop}\label{prop:3.13}
Assume that $\Gp  =0$. Then $A$ is the only $m$-dissipative extension of $A_0$. However, if $\Gm  \neq 0$, then there exists an infinite number of generators $B$ of a $C_0$-semigroup such that $A_0\subset B$.  
\end{prop}
\begin{proof}
One has 
\[   2\re \langle Au,u\rangle_V=\| \Gp   u\|_{\Hp}^2-\|\Gm  u\|_{\Hm}^2=-\|\Gm  u\|_{\Hm}^2\leq 0 \]
for all $u\in\dom(A)$.  Thus $A$ is dissipative. It follows from Theorem~\ref{th:2.3} that $A$ is maximal dissipative  and hence $m$-dissipative by Phillips'The\-o\-rem. If $\Gm  \neq 0$, then $\dom(\overline{A_0})=\ker \Gm  \neq \dom(A)$. 
 Thus it follows 
from \cite[A-II, Theorem 1.33 p. 46]{Na} that $A_0$ has infinitely many extensions generating a $C_0$-semigroup. 
\end{proof}
Thus in the case where $\Gp  =0$ and $\Gm  \neq 0$, only one of the infinitely many $C_0$-semigroups having an extension of $A_0$ as generator is contractive. Next we consider the case when $\Gm  =0$. 
 \begin{prop}\label{prop:3.14}
 If $\Gm  =0$, then $\overline{A_0}$ is the only $m$-dissipative extension  of $A_0$. 
\end{prop}   
\begin{proof}
One has $\Hm =\{0\}$. Thus $\Phi=0$ is the only contraction from $\Hm $ to $\Hp  $. Thus 
\begin{eqnarray*}
	\dom(A_\Phi) & = & \{ w\in \dom(A):0=\Gp  w\}\\
	 &  = & \ker \Gp  \\
	 & = & \ker \Gp  \cap \ker \Gm  \\
	  & = & \dom(\overline{A_0}).
\end{eqnarray*}
\end{proof}
As a consequence $\overline{A_0}$ is also the only extension of $A_0$ generating a $C_0$-semigroup. 

We conclude this section by establishing isomorphisms between boundary quadruples. 
\begin{thm}\label{th:3.16}
	Let $(\Hm, \Hp,\Gm,\Gp)$ be a boundary quadruple for $A_0$. Let $\tilde{\Gm}:\dom(A)\to \tilde{\Hm}$ and $\tilde{\Gp}:\dom(A)\to \tilde{\Hp}$ be linear where $\tilde{\Hm}$ and $\tilde{\Hp}$ are Hilbert spaces. The following assertions are equivalent:
	\begin{itemize}
		\item[(i)] $(\tilde{\Hm},\tilde{\Hp},\tilde{\Gm},\tilde{\Gp})$ is a boundary quadruple;
		\item[(ii)] \begin{enumerate}
			\item[a)] $ \tilde{\Gm}\dom(A)=\tilde{\Hm}$ and $\tilde{\Gp}\dom(A)=\tilde{\Hp}$; 
			\item[b)] there exists an isomorphism $\Psi:H\to\tilde{H}$ where 
			$H:=\Hm\oplus \Hp$ and  $ \tilde{H}:=\tilde{\Hm}\oplus \tilde{\Hp}$
			 such that 
		\[   \Psi(\Gm u,\Gp u)=(\tilde{\Gm}u,\tilde{\Gp}u)\mbox{ for all }u\in\dom(A)\]
		and \[  \Psi^*\tilde{C} \Psi=C,\]
		where $C(x_{\mbox{\tiny{-}}},x_+)=(-x_{\mbox{\tiny{-}}},x_+)$ for $x_{\mbox{\tiny{-}}}\in \Hm$, $x_+\in\Hp$ and $\tilde{C} (\tilde{x_{\mbox{\tiny{-}}}},\tilde{x_+})=(-\tilde{x_{\mbox{\tiny{-}}}},\tilde{x_+})$ for 
		$\tilde{x_{\mbox{\tiny{-}}}}\in\tilde{\Hm}$, $\tilde{x_+}\in\tilde{\Hp}$. 
		\end{enumerate} 
	\end{itemize} 
\end{thm} 
\begin{proof}
$(i)\Rightarrow (ii)$ Assume that $(\tilde{\Hm},\tilde{\Hp},\tilde{\Gm},\tilde{\Gp})$ is a boundary quadruple. Then $a)$  follows from Proposition~\ref{prop:3.5}. To prove $b)$, consider the mapping $G:=(\Gm,\Gp):\dom(A)\to H:=\Hm\oplus \Hp$ given by $G(u)=(\Gm u,\Gp u)$ and similarly $\tilde{G}:=(\tilde{\Gm},\tilde{\Gp}):\dom(A)\to \tilde{H}:=\tilde{\Hm}\oplus \tilde{\Hp}$. By Proposition~\ref{prop:3.8}, $\ker G=\ker \tilde{G}=\dom(\overline{A_0})$.  Thus $G_0:=G_{|\dom(A_0)^\perp}$ and $\tilde{G_0}:={\tilde{G}}_{|\dom(A_0)^\perp}$ define two isomorphisms $G_0:\dom(A_0)^\perp \to H$ and   $\tilde{G_0}:\dom(A_0)^\perp \to \tilde{H}$. Thus $\Psi:=\tilde{G_0}G_0^{-1}\in {\mathcal L}(H,\tilde{H})$ is an isomorphism. Since 
\[  \langle Au,v\rangle_V +\langle u,Av\rangle_V=0\] 
if $u\in \dom(\overline{A_0})$ or $v\in \dom(\overline{A_0})$, it follows that $\tilde{G}=\Psi G$.  Notice that 
\[  \langle Au,v\rangle_V +\langle u,Av\rangle_V=\langle CGu,Gv\rangle_H=\langle \tilde{C} \tilde{G}u,\tilde{G}v\rangle_{\tilde{H}}\]
	for all   $u,v\in\dom(A)$, where $Gu=(\Gm u,\Gp u)\in\Hm\oplus \Hp=H$. Similarly $\tilde{G}u$ is written as a vector. Let $u\in\dom(A)$. Then for all $v\in\dom(A)$, 
	\[  \langle CGu,Gv\rangle_H=\langle \tilde{C} \Psi {G}u,\Psi{G}v\rangle_{\tilde{H}} =\langle \Psi^*\tilde{C} \Psi {G}u,{G}v\rangle_{\tilde{H}}.  \]
	Since $G$ is surjective, it follows that $CGu =\Psi^*\tilde{C}\Psi  G u$ for all $u\in\dom(A)$. Consequently, $C=\Psi^*\tilde{C}\Psi$. \\
$(ii)\Rightarrow (i)$ This is shown by reversing the arguments leading to "$(i)$ implies $(ii)$". 	    
\end{proof}
\section{Unitary groups}\label{sec:4}
Let $V$ be a Hilbert space over $\K=\R$ or $\C$. In this section we want to determine all extensions of a densely defined skew-symmetric operator which generate a unitary group. We recall the following two well-known facts. An operator $B$ on $V$ generates a $C_0$-group $(U(t))_{t\in\R}$ if and only if 
$\pm B$ generate  $C_0$-semigroups $(U_\pm (t))_{t\geq 0}$. In that case, $U(t)=U_+(t)$ for $t\geq 0$ and $U(-t)=U_{-}(t)$ for $t<0$. Moreover, if $B$ generates a $C_0$-semigroup $(S(t))_{t\geq 0}$, then $B^*$ generates the $C_0$-semigroup $(S(t)^*)_{t\geq 0}$. By a \emph{unitary $C_0$-group} we mean a $C_0$-group of unitary operators. The following is well-known.
\begin{prop}\label{prop:4.1}
	Let $B$ be an operator on $V$.  The following assertions are equivalent:
	\begin{itemize}
		\item[(i)] $B$ generates a unitary $C_0$-group;
		\item[(ii)] $\pm B$ is $m$-dissipative;
		\item[(iii)] $\dom(B)$ is dense and $-B=B^*$. 
	\end{itemize}
\end{prop}     
\begin{proof}
$(ii)\Longleftrightarrow (i)$ $\pm B$ are $m$-dissipative if and only $\pm B$ generate contractive $C_0$-semigroups $(U_{\pm}(t))_{t\geq 0}$. This in turn is equivalent to $B$ generating a contractive hence  isometric group; i.e. a unitary $C_0$-group. \\
$(i)\Rightarrow (iii)$ Let $B$ be the generator of the unitary group $(U(t))_{t\in\R}$. Then $B^*$ generates the semigroup $(U(t)^*)_{t\geq 0}$ and $-B$ the $C_0$-semigroup $(U(-t))_{t\geq 0}$. Since $U(-t)=U(t)^*$ for all $t\geq 0$, it follows that $-B=B^*$. \\
$(iii)\Rightarrow (i)$ If $-B=B^*$, then $B$ is skew-symmetric and $\Gp  =\Gm  =0$ by Proposition~\ref{prop:3.12}. By Proposition~\ref{prop:3.13}, $-B^*=B$ generates a contractive $C_0$-semigroup. Applying the same result to $-B$ instead of $B$ we deduce that also $-B$ generates a contractive $C_0$-semigroup. Thus $B$ generates a unitary $C_0$-group.    
\end{proof}
Now we describe all generators of a unitary $C_0$-group which are extensions of  a given skew-symmetric operator. 

Recall that $V$ is a real or complex Hilbert space. Let $A_0$ be a densely skew-symmetric operator on $V$ and let $A=-A_0^*$. Let $(\Hm,\Hp,\Gm  ,\Gp  )$ be a  boundary quadruple for $A_0$. 
We let $\Hm =\ran \Gm  $, $\Hp  =\ran \Gp  $ and if $\Phi\in {\mathcal L}(\Hm ,\Hp  )$ is a contraction, we define $A_\phi\subset A$ as before on the domain
\[  \dom(A_\Phi):=\{ w\in\dom(A):\Phi \Gm  w=\Gp  w\}\mbox{ by }A_\Phi w:=Aw.\]
\begin{thm}\label{th:4.2}
Let $B$ be an operator on $V$ such that $A_0\subset B$. The following assertions are equivalent:
\begin{itemize}
	\item[(i)] $B$ generates a unitary $C_0$-group;
	\item[(ii)] there exists a unitary operator $\Phi\in{\mathcal L}(\Hm ,\Hp  )$ such that $B=A_\Phi$.   
\end{itemize} 
\end{thm} 
  \begin{proof}
 $(i)\Rightarrow (ii)$ Let $B$ be the generator of a unitary group. Then $\pm B$ are $m$-dissipative. By Theorem~\ref{th:3.9} there exists a contraction $\Phi:\Hm \to \Hp  $ such that $B=A_\Phi$. Apply Theorem~\ref{th:3.9} to $-B$ instead of $B$ and observe that $(\Hm,\Hp,\Gp  ,\Gm  )$ is a boundary quadruple for $-A_0$; here we just interchanged $\Gm  $ and $\Gp  $. Thus  we find a contraction $\Psi:\Hp  \to \Hm $ such that 
 \[ \dom(-B)=\{  w\in\dom(A):\Psi \Gp   w=\Gm  w\}.\]
Let $\xm    \in \Hm $. We claim that $\Psi \Phi \xm    =\xm    $. In fact, by the Interpolation Lemma (Lemma~\ref{lem:3.2}) there exists $w\in \dom(A)$ such that $\Gm   w=\xm    $ and  $\Gp   w=\Phi \xm$. Thus $\Phi (\Gm   w)=\Gp   w$ and so $w\in \dom(A_\Phi)=\dom(B)$. Hence $w\in \dom(-B)=\dom(B)$. Consequently,
\[   \Psi \Gp   w=\Gm   w\mbox{ i.e. }\Psi \Phi \xm    =\Psi \Gp   w=\Gm   w=\xm    .\]
This proves the claim. One  similarly shows that $\Phi \Psi \xp    =\xp    $ for all $\xp    \in \Hp  $. Thus $\Phi$ is unitary.  \\
$(ii)\Rightarrow (i)$ Assume that $\Phi:\Hm \to \Hp  $ is unitary and $B=A_\Phi$. In particular 
\[  \dom(B)=\{  w\in\dom(A):\Phi \Gm  w=\Gp  w\}=\{  w\in\dom(A):\Gm  w=\Phi^{-1}\Gp  w\}.\]
Thus $-B $ is $m$-dissipative by Theorem~\ref{th:3.9}. We have shown that $\pm B$ are $m$-dissipative. 
 \end{proof}
\begin{cor}\label{cor:4.3}
The following assertions are equivalent:
\begin{itemize}
	\item[(i)] there exists a generator $B$ of a unitary group such that $A_0\subset B$;
	\item[(ii)] $\Hm $ and $\Hp  $ are isomorphic as Hilbert spaces;
	\item[(iii)] $\ker (\Id -A)$ and $\ker (\Id + A)$ are isomorphic. 
\end{itemize}
\end{cor}
\begin{proof}
	$(i)\Longleftrightarrow (ii)$ follows from Theorem~\ref{th:4.2}.\\
	 $(ii)\Longleftrightarrow (iii)$ follows from Proposition~\ref{prop:3.7}.
\end{proof}
\section{Selfadjoint extensions of symmetric operators}\label{sec:5}
In this section we consider a complex Hilbert space $V$. An operator $S$ on $V$ is called \emph{symmetric} if 
\[  \langle Su,v\rangle_V=\langle u,Sv\rangle_V\mbox{ for all }u,v\in\dom(S).\]
If $S$ is densely defined, then $S$ is symmetric if and only if $S\subset S^*$. An operator $T$ on $V$ is called \emph{selfadjoint} if $\dom(T)$ is dense in $V$ and $T=T^*$. 

It is immediate that an operator $S$ is symmetric if and only if $iS$ is skew-symmetric.  

Selfadjointness can be characterized by Stone's Theorem.
\begin{thm}[Stone]\label{th:5.1}
Let $T$ be an operator on $H$. The following assertions are equivalent:
\begin{itemize}
	\item[(i)] $T$ is selfadjoint;
	\item[(ii)] $iT$ generates a unitary $C_0$-group.   
\end{itemize} 
\end{thm}
\begin{proof}
	Let $T$ be densely defined. Then $(iT)^*=-iT^*$. Thus $T$ is selfadjoint if and only if $(iT)^*=-iT$. By Proposition~\ref{prop:4.1}, this is equivalent to $(ii)$.  
\end{proof}
Theorem~\ref{th:5.1} allows us to transport our results from Section~\ref{sec:4} to the characterization of all selfadjoint extensions of a symmetric operator. To that aim, let $S$ be a densely defined symmetric operator on $V$. 
\begin{defn}\label{def:5.2}	
A \emph{boundary quadruple} $(\Hm,\Hp,,\Gm  ,\Gp  )$ for $S$ consists of  complex pre-Hilbert spaces $\Hm,\Hp$ and linear surjective maps $\Gm:\dom(S^*)\to \Hm$, $\Gp:\dom(S^*)\to \Hp$ such that  
\begin{equation}\label{eq:5.1}
	\langle S^* u,v\rangle_V -\langle u,S^* v\rangle_V =i\left(  \langle \Gm  u,\Gm   v\rangle_{\Hm}-\langle \Gp  u,\Gp  v\rangle_{\Hp} \right)  
\end{equation}	
for all $u,v\in \dom(S^*)$, and 
\begin{equation}\label{eq:5.1bis}
\ker \Gp   +\ker \Gm   =\dom(S^*).
\end{equation}
\end{defn}	 
\begin{prop}\label{prop:5.3}
A  boundary quadruple $(\Hm,\Hp,\Gm  ,\Gp)$ for $S$ always exists. Then $\Hm$ and $\Hp$ are complete and $\Gm  ,\Gp  $ are continuous if $\dom(S^*)$ carries the graph norm.  
\end{prop}
\begin{proof}
	The operator $A_0:=iS$ is skew-symmetric and $-A_0^*=iS^*$. Let $(\Hm,\Hp, \Gm  ,\Gp  )$ be a  boundary quadruple for $A_0$. Then 
	\begin{eqnarray*}
		\langle S^* u,v\rangle_V -\langle u,S^* v\rangle_V  & =  &-i \left(   \langle iS^*u,v\rangle_V+\langle  u,iS^*v\rangle_V\right)\\
		 & = & -i\left( \langle \Gp   u,\Gp   v\rangle_{\Hp}-\langle \Gm   u,\Gm  v\rangle_{\Hm} \right) 
	\end{eqnarray*}	
for all $u,v\in \dom(S^*)$. Now the existence follows from Example~\ref{ex:3.6} and Proposition~\ref{prop:3.7}. Proposition~\ref{prop:3.5} says that $\Hm $ and $\Hp  $ are complete and Lemma~\ref{lem:3.4} that $\Gm  ,\Gp  $ are continuous.   
\end{proof}
Let $(\Hm,\Hp,\Gp  ,\Gm  )$ be a boundary quadruple for $S$. Then we have the following two results. 
\begin{prop}\label{prop:5.4}
The following assertions are equivalent:
\begin{itemize}
	\item[(i)] $S$ has a selfadjoint extension;
	\item[(ii)]$\Hp  $ and $\Hm $ are  isomorphic, i.e. there exists a unitary operator from $\Hp  $ onto $\Hm $;
	\item[(iii)] $\ker (i\Id -S^*)$ and $\ker (i\Id +S^*)$ are isomorphic.    
\end{itemize}
\end{prop}
\begin{thm}\label{th:5.5}
Assume that $\Hp  $ and $\Hm $ are isomorphic. Let $T$ be an operator such that $S\subset T$. The following assertions are equivalent:
\begin{itemize}
	\item[(i)] $T$ is selfadjoint;
	\item[(ii)] there exists a unitary operator $\Phi:\Hp  \to \Hm $ such that
	\[   \dom(T)=\{  w\in \dom(S^*):\Phi \Gp   w=\Gm  w\},\,\, Tw=S^*w\]
	for all $w\in \dom(T)$.  
\end{itemize}
\end{thm}
\begin{proof}[Proof of Propostiion~\ref{prop:5.4} and Theorem~\ref{th:5.5}]
	Let $A_0=iS$. Then we have $A=-A_0^*=iS^*$. Thus Proposition~\ref{prop:5.4} follows from Corollary~\ref{cor:4.3}.   Let $S\subset T$. Then $iS\subset iT$. By Theorem~\ref{th:5.1}, $T$ is selfadjoint if  and only if $iT$ generates a unitary $C_0$-group. By Theorem~\ref{th:4.2}, this is equivalent to $(ii)$ of Theorem~\ref{th:5.5}
	\end{proof}	
Specified to the  boundary quadruple of Proposition~\ref{prop:3.5}, Theorem~\ref{th:5.5} is a variant of \cite[Theorem X.2]{RS75}. A boundary triple for $S$ as defined in \cite[Chapter 14]{Sch12} is essentially equivalent to a  boundary quadruple in which $\Hp$ and $\Hm  $ are isomorphic. Then Theorem~\ref{th:5.5} also follows from the equivalence of $(i)$ and $(iii)$ in \cite[Theorem 4.10]{Sch12}. 

\section{Examples}\label{sec:6}
Let $H$ be a Hilbert space over $\K=\R$ or $\C$ and let $-\infty\leq a<b\leq \infty$, 
\[  L^2((a,b), H):=\{  u:(a,b)\to H\mbox{ measurable; } \int_a^b \|u(t)\|_H^2 dt<\infty\}. \]  
Define the Sobolev space \\
$H^1((a,b),H) := \{ u\in L^2((a,b),H):\exists u'\in L^2((a,b),H)\mbox{ such that }\\
-\int_a^b \varphi'(t)u(t)dt = \int_a^b u'(t)\varphi(t) dt
  \mbox{ for all }\varphi\in {\mathcal C}_c^\infty ((a,b),\R)\}$. 
  
  Then $H^1((a,b),H)$ is a Hilbert space for the norm 
  \[  \|u\|_{H^1}^2=\|u\|_{L^2}^2 +\|u'\|_{L^2}^2.\]

\begin{lem}\label{lem:6.1}
	One has 
	\begin{itemize}
		\item[a)] $H^1((a,b),H)\subset {\mathcal C}([a,b],H)$ if $-\infty <a,b<\infty$.
		\item[b)] $H^1((a,b),H)\subset {\mathcal C}([a,\infty),H)$ if $a>-\infty$ and $\lim_{t\to\infty}u(t)=0$ for all $u\in H^1((a,\infty),H)$. 
		\item[c)] $H^1((-\infty,b),H)\subset \{  u\in {\mathcal C} ((-\infty,b],H):\lim_{t\to -\infty} u(t)=0\}$. 
		\item[d)] $H^1((-\infty,\infty),H)\subset\{ u\in{\mathcal C} ((-\infty,\infty),H):\lim_{t\to\pm\infty}  u(t)=0\}$. 
	\end{itemize}
\end{lem}
This can be proved as in the scalar case, see \cite[Theorem 8.2 and Corollary 8.9]{Br11} and \cite[Sec. 8.5]{AU22} for a vector-valued version. 

Here we identify each $u\in H^1((a,b),H)$ with the unique $\widetilde{u}\in {\mathcal C}([a,b]\cap \R,H)$ such that $\widetilde{u}(t)=u(t)$ a.e. The following integration-by-parts formula holds, see \cite[Corollary 8.9 and 8.10]{Br11}.
\begin{lem}\label{lem:6.2}
Let $u,v\in H^1((a,b), H)$. Then $\langle u'(\cdot),v(\cdot)\rangle_H,\,\, \langle u(\cdot),v'(\cdot)\rangle_H\in L^1((a,b),H)$ and 
\[  \int_a^b \langle u'(t),v(t)\rangle_H dt + \int_a^b\langle u(t),v'(t)\rangle_H dt \]
\[=\begin{cases}
	  \langle u(b),v(b)\rangle_H -\langle u(a),v(a)\rangle_H \mbox{ if } -\infty<a<b<\infty;\\
	  \langle u(b),v(b)\rangle_H \mbox{  if }-\infty=a<b<\infty;\\
	  -\langle u(a),v(a)\rangle_H \mbox{ if }-\infty<a<b=\infty;\\
	  0\mbox{ if }a=-\infty\mbox{ and }b=\infty. 
\end{cases}\]
\end{lem}   
\begin{exam}\label{ex:6.3}
Let $V=L^2((a,b),H)$, $\dom(A_0)={\mathcal C}_c^\infty ((a,b),H)$, $A_0u =u'$. Then $A_0$ is skew-symmetric with dense domain. Let $B$ be an operator on $V$ such that $A_0\subset B$. \\
\begin{itemize}
	\item[a)] Assume that $-\infty <a<b<\infty$. Then $B$ is $m$-dissipative if and only if there exists  a contraction $\Phi\in{\mathcal L}(H)$ such that 
	  \[\dom(B)=\{ u\in H^1((a,b),H):\Phi u(a)=u(b) \}.\]
	  Moreover $B$ generates a unitary $C_0$-semigroup if and only if $\Phi$ is unitary.  
	\item[b)] Let $-\infty=a<b<\infty$. Then $B$ is $m$-dissipative if and only if $\dom(B)=\{u\in H^1(-\infty,b),H):u(b)=0\} $. In that case the semigroup $(S(t))_{t\geq 0}$ generated by $B$ is given by $(S(t)f)(x)=f(x+t)$ if $x+t\leq b$ and $(S(t)f)(x)=0$ if $x+t> b$. This $B$ is the only generator of a $C_0$-semigroup which is an extension of $A_0$. 
	\item[c)] Let $-\infty<a<b=\infty$. Then $B$ is $m$-dissipative if and only if $\dom(B)=H^1((a,\infty),H)$. In that case the semigroup $(S(t))_{t\geq 0}$ generated by $B$ is given by 
	\[  (S(t)f)(x)=f(x+t)\mbox{ for }x\in (a,\infty),\, t>0,\, f\in V.\]
	There are infinitely many other non-contractive $C_0$-semigroups having as generator an extension of $A_0$.  
	\item[d)] $-\infty=a,b=\infty$. Then $\dom(\overline{A_0})=H^1(\R,H)$ and $\overline{A_0}$ generates a unitary $C_0$-group given by $(U(t)f)(x)=f(x+t)$.  
\end{itemize}
\begin{proof}
A boundary quadruple $(\Hm,\Hp,\Gm  ,\Gp )$ is given by $\Hm=\Gm\dom(A)$ and $\Hp=\Gp\dom(A))$ where : \\
$\Gp   u=u(b)$ and $\Gm  u=u(a)$ in the case $a)$,\\
$\Gm  =0$ and  	$\Gp   u=u(b)$  in the case $b)$, \\
$\Gp  =0$ and  	$\Gm   u=u(a)$  in the case $c)$, \\
$\Gm  =\Gp  =0$ in the case $d)$. 

Now $a)$ follows from Theorem~\ref{th:3.9}, $b)$ from Proposition~\ref{prop:3.13}, $c)$ from Proposition \ref{prop:3.12} and $d)$ from Theorem~\ref{th:5.5}.  
\end{proof}
\end{exam}
Next we give an example where arbitrary closed subspaces $H_1,H_2$ of $H$ occur as $\Gm   \dom(A)$ and $\Gp  \dom(A)$.  
\begin{exam}\label{ex:6.4}
	Let $H$ be a Hilbert space over $\K=\R$ or $\C$, and let $H_1,H_2$ be closed subspaces of $H$. We construct  a densely defined symmetric operator $A_0$ on a  space $V$ and  $(\Hm,\Hp,\Gm  ,\Gp  )$, a boundary quadruple, such that $\Hp  =H_2$, $\Hm =H_1  $. For that we let $-\infty<a<c<\infty$, 
	\[V= \{   u\in L^2((a,b),H):u(t)\in H_1,\,t\in (a,c)\mbox{ and }u(t)\in H_2,\,t\in (c,d) \}.     \]  
	This is a closed subspace of $L^2((a,b),H)$. Define the operator $A_0$ on $V$ by 
	\[ \dom(A_0):=\{ u\in H_0^1((a,b),H):u(t)\in H_1,\,t\in [a,c], \,\, u(t)\in H_2,\,\, t\in [c,b]  \}.\]
	Thus $u\in{\mathcal C}([a,b],H)$ and $u(a)=u(b)=0$, $u(c)\in H_1\cap H_2$ for each $u\in \dom(A_0)$. Define $A_0 u=u'\in V$. Then $\dom(A_0)$ is dense since 
	\[H_0^1((a,c),H_1)\oplus H_0^1((c,d),H_2)\subset \dom(A_0)\]
	 and $H_0^1((a,c),H_1)$ is dense in $L^2((a,c),H_1)$ and $H_0^1((c,b),H_2)$ is dense in $L^2((c,b),H_2)$. The operator $A_0$ is symmetric. This is a consequence of Lemma~\ref{lem:6.2} since $\lim_{x\to c,x<c}v(x)=\lim_{x\to C,x>c} v(x)$ for all $v\in\dom(A_0)$. Moreover, $A=-A_0^* $ is given by: 
	 \[ \dom(A)=\{  u\in H^1((a,b),H):u(t)\in H_1, \, t\leq c\mbox{ and }u(t)\in H_2,\, t\geq c\}\] and $Au=u'$. 
	 
	 Define $\Gp  ,\Gm  :\dom(A)\to H$ by $\Gp  u=u(b)$ and $\Gm  u=u(a)$. Then we see that $(\Gm \dom(A),\Gp\dom(A),\Gm  ,\Gp  )$ is a  boundary quadruple for $A_0$ and $\Gp  \dom(A)=H_2$, $\Gm  \dom(A)=H_1$. We omit the details of the proof.     
	 \end{exam}
Next we consider the second derivative. Let $-\infty<a<b<\infty$, $V=L^2((a,b),H)$. By 
\[ H^2((a,b),H):=\{    u\in H^1((a,b),H):u'\in H^1((a,b),H)\} \]
we define the Sobolev space of second order.  
 It follows from Lemma~\ref{lem:6.1} that 
\[  H^2((a,b),H)\subset {\mathcal C}^1([a,b],H).\]
We define the operators $\Gm,\Gp\in {\mathcal L}(H^2(a,b;H),H\times H )$ by 
\begin{equation}\label{eq:6.1}
	\Gm u=\frac{\sqrt{2}}{2} (u(a)-iu'(a)), u(b)+iu'(b))\in H\times H
\end{equation} 	
and 
\begin{equation}\label{eq:6.2}
	\Gp u=\frac{\sqrt{2}}{2} (u(a)+iu'(a)), u(b)-iu'(b))\in H\times H.
\end{equation} 	
We define the symmetric operator $S$ on $V=L^2(a,b;H)$ by 
\[Sv=v''\mbox{ and }\dom(S):={\mathcal C}_c^\infty(a,b;H).\]
\begin{thm}\label{th:6.1}
	Let $T$ be an operator such that 
	\[  {\mathcal C}_c^\infty ((a,b),H)\subset \dom(T)\subset H^2((a,b),H)\mbox{ and }Tv=v'',\,\,v\in\dom(T).\]
	The following assertions  are equivalent:
	\begin{itemize}
		\item[(i)] $iT$ is $m$-dissipative;
		\item[(ii)]there exists a contraction $\Phi\in {\mathcal L}(H\times H)$ such that
		\begin{equation}\label{eq:6.3}
			\dom(T)=\{ v\in H^2((a,b),H):\Gp v=\Phi \Gm v\}.
		\end{equation}  
	\end{itemize}
\end{thm}
\begin{thm}\label{th:6.2}
	Let $T$ be an operator on $V$ such that $S\subset T$. The following assertions are equivalent:
	\begin{itemize}
		\item[(i)] $T$ is self-adjoint;
		\item[(ii)] there exists a unitary operator $\Phi\in {\mathcal L}(H\times H)$ such that 
		\[  \dom(T)=\{  v\in H^2((a,b),H): \Phi \Gm v=\Gp v\}\mbox{ and } Tv=v'',\,\, v\in \dom(T). \]  
	\end{itemize}
\end{thm}
We find the three classical cases.\\
a) \emph{Neumann boundary conditions:} $v'(a)=v'(b)=0$ for $\Phi=\Id$. \\
b) \emph{Dirichlet boundary conditions:} $v(a)=v(b)=0$ for $\Phi=-\Id$.\\
c) \emph{Periodic boundary conditions:} $v(a)=v(b)$, $v'(a)=v'(b)$ for \[\Phi=\begin{pmatrix}
	0 & \Id_H\\
	\Id_H & 0 
\end{pmatrix} .\]
In each of these three cases the operator $T$ given by $Tv=v''$ with the domain \eqref{eq:6.3} is self-adjoint. \\
d) \emph{Robin boundary conditions:} $(\Id_H-\Phi_1)u(a)=i(\Id_H +\Phi_2)u'(a)$, \\$(\Id_H-\Phi_2)u(b)=i(\Id_H -\Phi_2)u'(b)$ for  \[\Phi=\begin{pmatrix}
	\Phi_1 & 0\\
	0 & \Phi_2 
\end{pmatrix} . \]
If $\Phi_1,\Phi_2\in {\mathcal L}(H)$ are contractions then the operator $iT$ with domain  \eqref{eq:6.3} given by 
$iTv=iv''$ is $m$-dissipative. If $\Phi_1,\Phi_2\in {\mathcal L}(H)$ are unitary, then the operator $T$ with domain \eqref{eq:6.3} given by $Tv=v''$ is self-adjoint.  

For the proof of Theorem~\ref{th:6.1} and Theorem~\ref{th:6.2} we need the following.
\begin{lem}\label{lem:6.3}
	Let $u\in L^2((a,b),H)$ such that $u''\in L^2((a,b),H)$; i.e. 
	\[  \int_a^b u(t)w''(t)dt =\int_a^bu''(t)w(t)dt\mbox{ for all }w\in {\mathcal C}_c^\infty((a,b),\R).\]
	Then $u\in H^2((a,b),H)$.  
\end{lem} 
\begin{proof}
	We fix $\Psi\in {\mathcal C}_c^\infty((a,b),\R)$ such that $\int_a^b\Psi(t)dt=1$. Let $w\in {\mathcal C}_c^\infty((a,b),\R)$. Then $\widetilde{w}:=w-\left( \int_a^b w(t)dt\right) \Psi\in {\mathcal C}_c^\infty((a,b),\R)$ and 
	$\int_a^b \widetilde{w}(t)dt=1$.  Thus $\varphi(t):=\int_a^t \widetilde{w}(s)ds$ defines $\varphi\in {\mathcal C}_c^\infty((a,b),\R)$ and $\varphi'=\widetilde{w}$, $w'=\varphi'' + \left( \int_a^b w(t)dt\right) \Psi'$. It follows from 
	the hypothesis that 
	\begin{eqnarray*}
		\int_a^b u(t)w'(t)dt & = & \int_a^b u''(t)\varphi(t)dt + \int_a^b w(t)dt\int_a^bu(t)\Psi'(t)dt\\
		& = & \int_a^b u''(t)\int_a^tw(s)ds -\left( \int_a^bu''(t)\Psi(t)dt\right) \int_a^b w(t)dt  \\
		&  & 	 + \int_a^b w(t)dt \int_a^b u(t)\Psi'(t)dt\\
		& = & \int_a^b \int_s^b u''(t)dt w(s)ds + \int_a^b cw(s)ds,
	\end{eqnarray*} 
	where $c=-\int_a^b u''(t)\Psi(t)dt + \int_a^b u(t)\Psi'(t)dt$.  Since $w\in {\mathcal C}_c^\infty(a,b;\R)$ is arbitrary, this implies that $u\in H^1((a,b),\C)$ and 
	\[ - u'(s)=c +\int_s^b u''(t)dt\mbox{ for  almost all $s$. }\]
	Thus $u'\in H^1((a,b),H)$ and consequently $u\in H^2((a,b),\C)$. 
\end{proof}
\begin{rem}
	Lemma~\ref{lem:6.3} is false in higher dimensions. Indeed, there exists $u\in{\mathcal C}(\overline{\D})\cap {\mathcal C}^2(\D)$ such that $\Delta u=0$ but $u\notin H^1(\D)$, where $\D$ is the open unit disc in $\R^2$. See 
	\cite[Example 6.68]{AU22}, a famous example due to Hadamard. 
\end{rem}
\begin{proof}[Proof of Theorem~\ref{th:6.1} and ~\ref{th:6.2}] 
	It follows from Lemma~\ref{lem:6.3} that \[\dom(S^*)=H^2((a,b),H)\] and $S^*u=u''$ for all $u\in H^2((a,b),H)$. For $u,v\in H^2((a,b),H)$, integrating by parts yields 
	\begin{eqnarray*}
		b(u,v) & :=& i\left( \langle S^* u,v\rangle_V -\langle u, S^* v\rangle_V\right)\\
		& = & i\left( \int_a^b \langle u''(t)',v(t)\rangle_H dt)  -\int_a^b\langle u(t),v''(t)\rangle_H dt\right)\\
		& = & i\left( \langle u'(b),v(b)\rangle_H -\langle u'(a),v(a)\rangle_H-\langle u(b),v'(b)\rangle_H\right) +\\
		&  &  i\langle u(a),v'(a)\rangle_H\\
		&  = & \langle \Gp u,\Gp v\rangle_{H\times H} - \langle \Gm u,\Gm v\rangle_{H\times H},
	\end{eqnarray*}
	where $\Gm, \Gp\in {\mathcal L}(H^2((a,b),H),H\times H)$ are defined by \eqref{eq:6.1} and 	\eqref{eq:6.2}. This is a computation that we omit.  We show that 
	\begin{equation}\label{eq:6.4}
		\ker \Gm +\ker \Gp =H^2((a,b),H).
	\end{equation} 
	Let $u\in H^2((a,b),H)$. Choose $\Psi,\Phi\in {\mathcal C}^2([a,b],H)$ such that 
	\[   \varphi(a)=-iu'(a),\,\, \varphi'(a)=iu(a),\,\, \varphi(b)=\varphi'(b)=0\]
	and 
	\[ \Psi(b)=-iu'(b),\,\, \Psi'(b)=iu(b),\,\, \Psi(a)=\Psi'(a)=0.\]
	Then 
	\[ \frac{1}{2} (u+\varphi-\Psi)\in \ker \Gm\mbox{ and }\frac{1}{2} (u-\varphi+\Psi)\in \ker \Gp . \]
	Thus the sum $u$ is in $\ker \Gm+\ker \Gp$. Now, Theorem~\ref{th:6.1} is a direct consequence of Theorem~\ref{th:3.9}, and Theorem~\ref{th:6.2} follows from Theorem~\ref{th:5.5}.  	
\end{proof}  
\section{The wave equation}\label{sec:6,5}
In this section we treat the wave equation in terms of quadruples. We are most grateful to Nathanael Skrepek who informed us on the papers \cite{KZ,KZ18} by Kurula and Zwart,    where boundary triples are used for similar, but different results. We refer to \cite{Skr} for further results on the Maxwell equations. Let $\Omega\subset \R^d$ be a bounded, open set with Lipschitz boundary. We consider the skew-symmetric operator $A_0$ on $V:=L^2(\Omega)\times L^2(\Omega)^d$ given by 
\[  \dom(A_0) =\{  (u_1,u_2):u_1\in {\mathcal C}_c^\infty(\Omega),u_2\in {\mathcal C}_c^\infty (\Omega)^d\} \]
 and
 \[ A_0(u_1,u_2)=(\di u_2,\nabla u_1). \]
 Then $A:=(-A_0)^*$ is given by 
 \[\dom(A)=H^1(\Omega)\times H_{\di}(\Omega),\,\, A(u_1,u_2)=(\di u_2,\nabla u_1),\]
where $H_{\di}(\Omega):=\{  u\in L^2(\Omega):\di u\in L^2(\Omega)\}. $ 
By $\Gamma:=\partial \Omega$ we denote the boundary of $\partial \Omega$ and by $L^2(\Gamma)$ the Lebesgue space with respect to the surface measure. There exists a unique operator $\tr \in {\mathcal L}(H^1(\Omega),L^2(\Gamma)) $ such that $\tr u=u_{|\Gamma}$ if $u\in H^1(\Omega)\cap {\mathcal C}(\overline{\Omega})$. We write  $u_{\Gamma}:=\tr u$ for all $u\in H^1(\Omega)$, and call $u_{\Gamma}$ the \emph{trace} of $u$. Then $\ker \tr =H_0^1(\Omega)$, the closure of ${\mathcal C}_c^\infty (\Omega)$ in $H^1(\Omega)$. 

The space   $H^{1/2}(\Gamma):=\tr H^1(\Omega)$ is a Hilbert space for the  following norm. Let $g\in H^{1/2}(\Gamma)$. Then there exists a unique 
$u\in (\ker \tr )^\perp={H_0^1(\Omega)}^\perp$ such that $u_\Gamma=g$. We let $\|g\|_{H^{1/2}(\Gamma)}:=\|u\|_{H^1(\Omega)}$. Thus 
	\[  \|g\|_{H^{1/2}(\Gamma)}=\inf  \{  \|v\|_{H^1(\Omega)} :v\in H^1(\Omega), \tr v=g\}. \]   
Let $u\in H_{\di} (\Omega)$. Then 
	 
	 \[   \int_\Omega \di u\overline{v} + \int_\Omega u\cdot \overline{\nabla v} =0\]
	 for all $v\in H^1_0(\Omega)$. Thus there exists a unique functional $\nu \cdot u\in H^{-1/2}(\Gamma)=H^{1/2}(\Gamma) '$ defined by  
\[ \langle \nu \cdot u,v_{\Gamma} \rangle:=\int_{\Omega} \di u \overline{v} + \int_{\Omega} u\cdot \overline{\nabla v},\]
for all $v\in H^1(\Omega)$. Here $\nu$ stands symbolically for the outer normal (which is a function in $L^\infty (\Gamma)$).  If $u\in H^1(\Omega)$ such that $\Delta u\in L^2(\Omega)$, then $\nabla u\in H_{\di}(\Omega)$ and we let $\partial_{\nu} u:=\nu \cdot \nabla u\in H^{-1/2}(\Omega)$. Thus 
\[  \langle \partial_{\nu} u,v_{\Gamma}\rangle =\int_{\Omega}\Delta u \,v +\int_{\Omega}\nabla u \, \nabla v  \]
for all $v\in H^1(\Omega)$.  

Denote by $R:H^{-1/2}(\Gamma)\to H^{1/2}(\Gamma)$ the Riesz isomorphism defined by 
\[  \langle \varphi , v_{\Gamma}\rangle_{H^{-1/2},H^{1/2}}=\langle R\varphi,v_{\Gamma}\rangle_{H^{1/2}} \]
for all $\varphi\in  H^{-1/2}(\Gamma)$, $v\in H^1(\Gamma)$. 

It is well-known and easy to see that the continuous, linear mapping
\begin{equation}\label{eq:6.5.1}
	u\in H_{\di}(\Omega)\mapsto R(\nu\cdot u)\in H^{1/2}(\Gamma)
\end{equation} 	 
is surjective. Now we can formulate the main result of this section. 
 \begin{thm}\label{th:6.5.1}
 Let $\Hm=\Hp:=H^{1/2}(\Gamma)$ and define \[ \Gm,\Gp:\dom(A)=H^1(\Omega)\times H_{\di}(\Omega)\to H^{1/2}(\Gamma)\]
  by  
  \[ \Gp (u_1,u_2)=\frac{1}{2} u_{1 \Gamma} + R(\nu\cdot u_2)\mbox{ and } \Gm (u_1,u_2)=\frac{1}{2} u_{1 \Gamma} - R(\nu\cdot u_2). \]
  Then $(\Hm,\Hp,\Gm,\Gp)$ is a boundary quadruple for $A_0$. 
 \end{thm}
Let $\Phi\in {\mathcal L}(H^{1/2}(\Gamma))$ be a contraction. We define the domain of the operator $B_\Phi$ on $V=L^2(\Omega)\times L^2(\Omega)^d$ 
by
\[  \dom(B_{\Phi})=\{  (u_1,u_2)\in H^1(\Omega)\times H_{\di}(\Omega):\Phi(\Gm(u_1,u_2))=\Gp(u_1,u_2)\},\]
where $ \Gm (u_1,u_2)=\frac{1}{2} u_{1 \Gamma} - R(\nu\cdot u_2)$ and $\Gp (u_1,u_2)=\frac{1}{2} u_{1 \Gamma} + R(\nu\cdot u_2)$. 
\begin{cor}\label{cor:6.5.2}
	Let $B$ be an operator on $V$ such that $A_0\subset B$. The following assertions are equivalent:
	\begin{itemize}
		\item[(i)] $B$ is $m$-dissipative;
		\item[(ii)] there exists a contraction $\Phi\in {\mathcal L}(H^{1/2}(\Gamma))$ such that $B=B_\Phi$.  
	\end{itemize}
\end{cor}
In view of Theorem~\ref{th:3.9}, Corollary~\ref{cor:6.5.2} is a direct consequence of Theorem~\ref{th:6.5.1}. 
\begin{proof}[Proof of Theorem~\ref{th:6.5.1}]
It follows from the definition  of $\nu\cdot u_2$ for $u_2\in H_{\di} (\Omega)$ that for $u=(u_1,u_2)$, $v=(v_1,v_2)\in \dom(A)=H^1(\Omega)\times H_{\di}(\Omega)$, 
\begin{eqnarray*}  
	b(u,v) & = & \langle \nu\cdot u_2,v_{1\Gamma}\rangle_{H^{-1/2},H^{1/2}} + \langle \nu\cdot \overline{v_2},\overline{u_{1\Gamma}}\rangle_{H^{-1/2},H^{1/2}} \\
	 & = & \langle R(\nu\cdot u_2),v_{1\Gamma}\rangle_{H^{1/2}(\gamma)} + \langle u_{1\Gamma}, R(\nu\cdot {v_2}\rangle_{H^{1/2}(\Gamma)} \\
	  & = & \langle \Gp u,\Gp v\rangle_{H^{1/2}(\Gamma)} -\langle \Gm u,\Gm v\rangle_{H^{1/2}(\Gamma)}.
\end{eqnarray*}	
It remains to show that $G:=(\Gm,\Gp):\dom(A)\to H^{1/2}(\Gamma)\times H^{1/2}(\Gamma)$ is surjective. 

Let $\hm,\hp\in H^{1/2}(\Gamma)$. Using \eqref{eq:6.5.1} we find $u_1\in H^1(\Omega)$ such that $u_{1\Gamma}=\hm+\hp$, and 
$u_2\in H_{\di}(\Omega)$ such that $R(\nu\cdot u_2)=\frac{1}{2}(\hp -\hm)$. Thus 
\[  \Gp (u_1,u_2) =\frac{1}{2} u_{1\Gamma} + R(\nu\cdot u_2)=\frac{1}{2}(\hp +\hm) + \frac{1}{2}(\hp -\hm)=\hp \]
and 
\[  \Gm (u_1,u_2) =\frac{1}{2} u_{1\Gamma} - R(\nu\cdot u_2)=\frac{1}{2}(\hp +\hm) - \frac{1}{2}(\hp -\hm)=\hm. \]
By Remark~\ref{rem:3.3}, $(\Hm,\Hp,\Gm,\Gp)$ is a boundary quadruple. 
\end{proof}
We want to express Corollary~\ref{cor:6.5.2} in terms of a well-posedness result. Let $B:=B_\Phi$ where $\Phi:H^{1/2}(\Gamma)\to H^{1/2}(\Gamma)$ is a contraction. We consider the wave equation
\begin{equation}\label{eq:6.5.2}
\ddot{u}=\Delta u.
\end{equation}   	
For our purposes we call $u$ a \emph{weak solution} if $u\in {\mathcal C}(\R_+,L^2(\Omega))$ and 
\[  \frac{d^2}{dt^2}\langle u(t),v\rangle_{L^2(\Omega)} =\langle u(t),\Delta v\rangle_{L^2(\Omega)} =\langle u(t),\Delta v\rangle_{L^2(\Omega)}\]
	 for all $v\in {\mathcal C}^\infty_c(\Omega)$. 
	 \begin{cor}\label{cor:6.5.3}
	 	Let $(u_{01},u_{02})\in \dom(B)$. Then there exists a unique $u\in {\mathcal C}^1(\R_+,L^2(\Omega))\cap {\mathcal C}(\R_+,H^1(\Omega))$ such that 
	 	\begin{equation}
	 	\label{eq:6.5.3a}
	 	\ddot{u}=\Delta u\mbox{  weakly}
	 	\end{equation}
	 	 \begin{equation}
	 	 \label{eq:6.5.3b}
	 	 u(0)=u_{01},\, \dot{u}(0)=\di u_{02}\mbox{ and }
	 	 \end{equation}
	 	 \begin{equation}
	 	 \label{eq:6.5.3c}
	 	  \Phi \left(  \frac{1}{2} u(t)_{\Gamma}-R\left(\nu\cdot u_{02} + \partial_{\nu}\int_0^t u(s)ds\right)\right) =\frac{1}{2} u(t)_{\Gamma} +R\left(\nu\cdot u_{02} + 
	 	  \partial_{\nu}\int_0^t u(s)ds\right).
	 	 \end{equation}
	 \end{cor}
 Note that by \eqref{eq:6.5.3a}, 
 \[\Delta \int_0^t u(s) ds =\dot{u}(t)-\dot{u}(0)= \dot{u}(t)-\di u_{02}\in L^2(\Omega).\]
 Therefore $\partial_{\nu} \int_0^t u(s)ds\in H^{-1/2}(\Gamma)$. 
 \begin{proof}
 We first prove the existence of a solution. Let 
 \[w(t):=(u_1(t),u_2(t)):=T(t) (u_{01},u_{02}).\]
  Then $w\in {\mathcal C}^1(\R_+, L^2(\Omega)\times L^2(\Omega)^d)\cap {\mathcal C}(\R_+,\dom(B))$ and $\dot{w}(t)=Bw(t)$ for $t\geq 0$. 	Thus \[u_1\in {\mathcal C}^1(\R_+, L^2(\Omega))\cap {\mathcal C}(\R_+,H^1(\Omega)),\]
  \[ u_2\in {\mathcal C}^1(\R_+, L^2(\Omega)^d)\cap {\mathcal C}(\R_+,H_{\di}(\Omega))\] and 
  \[  \dot{u}_1(t)=\di u_2(t),\, \dot{u}_2(t)=\nabla u_1(t).\]
  Thus, for $v\in {\mathcal C}_c^\infty(\Omega)$, 
  \begin{eqnarray*}
  \frac{d^2}{dt^2} \langle u_1(t),v\rangle_{L^2(\Omega)} & = & \frac{d}{dt} \langle \dot{u}_1(t),v\rangle_{L^2(\Omega)}
    = 	\frac{d}{dt} \langle \di {u}_2(t),v\rangle_{L^2(\Omega)}\\
    &= & -\langle \nabla  u_1(t),\nabla v\rangle_{L^2(\Omega)^d}
      =  \langle u_1(t),\Delta v\rangle_{L^2(\Omega)}\\
       &= & -\frac{d}{dt}\langle u_2(t), \nabla v\rangle_{L^2(\Omega)}. 
  \end{eqnarray*}
Thus $u_1$ is a weak solution of \eqref{eq:6.5.3a}. Since $T(0)(u_{01},u_{02})=(u_{01},u_{02})$, \eqref{eq:6.5.3b} holds. Since $(u_1(t),u_2(t))\in\dom(B)$ we have 
\[  \Phi\left(     \frac{1}{2} u_1(t)_{\Gamma} -R(\nu\cdot u_2(t))   \right)=\frac{1}{2} u_1(t)_{\Gamma} +R(\nu\cdot u_2(t)).  \]
Observe that 
\[  \nu\cdot u_2(t)=\nu\cdot \left( u_{02} +\int_0^t \nabla u_1(s)ds \right) =\nu\cdot u_{02} +\partial_{\nu} \int_0^t u_1(s)ds .\]
This shows that $u:=u_1$ satisfies \eqref{eq:6.5.3c}. 

To prove the uniqueness, let $u$ be  a solution of \eqref{eq:6.5.3a},    \eqref{eq:6.5.3b} and \eqref{eq:6.5.3c}. Define $u_1=u$ and $u_2(t)=\int_0^t \nabla u(s)ds + u_{02}$. Then $w(t)=(u_1(t),u_2(t))\in\dom(B)$, $w\in {\mathcal C}^1(\R_+, L^2(\Omega)\times L^2(\Omega)^d)$ and $w$ satisfies \eqref{eq:6.5.3a},    \eqref{eq:6.5.3b} and \eqref{eq:6.5.3c}. Thus $w(t)=T(t) (u_{01},u_{02})$.  
 \end{proof}	
\begin{rem}	\label{rem:6.5.4}
Let $u_{01}\in H_0^1(\Omega)$, $u_{02}=0$, $(u^1(t),u^2(t)):=T(t)(u_{01},0)$, $w(t):=\int_0^t u^1(s) ds$. Then $w\in {\mathcal C}^2(\R_+,L^2(\Omega))\cap {\mathcal C}^1(\R_+,H^1(\Omega))$, $w(t)\in \dom(B)$, $\ddot{w}(t)=\Delta w$, $w(0)=0$, $\dot{w}(0)=u_{01}$. Since the semigroup is contractive, 
 \[  E(t):=\|\dot{w}(t)\|^2_{L^2(\Omega)} +\|\nabla w(t)\|^2_{L^2(\Omega)} = \|u_1(t)\|^2_{L^2(\Omega)} +\|u_2(t)\|^2_{L^2(\Omega)}   \]
 is decreasing. If $\Phi$ is unitary, then $E(t)$ is constant, the sum of cinetic and potential energy.     
\end{rem}
\section{Relation to the literature}\label{sec:7}
In the monograph \cite{GG} by V. I. and M. L. Gorbachuk and \cite{Sch12} by Schm\"udgen, boundary triples are used to parametrize seladjoint extensions of a densely defined symmetric operator. Wegner \cite{Weg} uses them to investigate $m$-dissipative extensions of a densely defined skew-symmetric operator--as we do in the present article--however using boundary quadruples. We now explain the relation of our results to those of Wegner and those presented in the two monographies \cite{GG} and \cite{Sch12}. 

We start with a general algebraic property of triples. 
\begin{lem}\label{lem:7.1}
	Let $W$ be a vector space, $H$ a pre-Hilbert space and 
	\[  \Gm  ,\Gp  , G_1, G_2:W\to H\mbox{ linear mappings.}\]  
	\begin{itemize}
		\item[a)] 
 	Then 
	\begin{equation}\label{eq:7.1}
		\Gm  =\frac{\sqrt{2}}{2} (G_2-G_1),\,\, \Gp  =\frac{\sqrt{2}}{2} (G_2+G_1),
	\end{equation}
if and only if 
	\begin{equation}\label{eq:7.2}
	G_1=\frac{\sqrt{2}}{2} (\Gp  -\Gm  ),\,\, G_2=\frac{\sqrt{2}}{2} (\Gp  +\Gm  ),
\end{equation}
\item[b)] If \eqref{eq:7.1} ($\Longleftrightarrow$ \eqref{eq:7.2}) holds, then 
\begin{equation}\label{eq:7.3}
\langle G_1u,G_2v\rangle_H + 	\langle G_2u,G_1v\rangle_H
= \langle \Gp  u,\Gp  v\rangle_H - \langle \Gm  u,\Gm  v\rangle_H
\end{equation}
for all $u,v\in W$. 
\item[c)] Let $\Hm =\Gm   W$, $\Hp  =\Gp   W$, $H_1=G_1 W$, $H_2=G_2 W$.  Then 
\[(\Gm  , \Gp  ):W\to \Hm \times \Hp   \mbox{ is surjective }\]
if and only if 
\[(G_1,G_2):W\to H_1\times H_2 \mbox{ is surjective.}\] 
 \end{itemize}
\end{lem}

Recall from Section~\ref{sec:3} that 
\begin{equation}\label{eq:7.4}
	(\Gm  ,\Gp  ):W\to \Hm \times \Hp   \mbox{ is surjective if and only if }
\end{equation}  
\begin{equation}\label{eq:7.5}
 \ker \Gm   +\ker \Gp  = W.
\end{equation}  
The proof of Lemma~\ref{lem:7.1} is straightforward.  

Now let $A_0$ be a skew-symmetric operator on $V$ with dense domain. Let $V$ be a complex Hilbert space. Following Wegner \cite[Defintion 4.1]{Weg}, a \emph{boundary triple} $(K,\Gamma_1,\Gamma_2)$ for $A_0$ consists of a Hilbert space $K$ and linear mappings 
$\Gamma_1,\Gamma_2:\dom(A)\to K$ such that 
\begin{equation}\label{eq:7.6}
-\langle Au,v\rangle_V-\langle u,Av\rangle_V=\langle \Gamma_1 u,\Gamma_2 v\rangle_H + \langle \Gamma_2 u,\Gamma_2 v\rangle_H	
\end{equation}	
for all $u,v\in\dom(A)$ and 
\begin{equation}\label{eq:7.7}
(\Gamma_1,\Gamma_2):\dom(A)\to K \mbox{ is sujective.} 	
 \end{equation}	
Let $H=K$, $G_1=\Gamma_2$, $G_2=\Gamma_1$ and define $\Gp  ,\Gm  $ by \eqref{eq:7.1}. Then $(H,H,\Gm  ,\Gp  )$ is a boundary quadruple (in the sense of our definition) with $\Gm   W=\Gp  W=H$. 

By Theorem~\ref{th:3.9} an operator $B\supset A_0$ is $m$-dissipative of and only if there exists a contraction $\Phi\in{\mathcal L}(H)$ such that 
\[  \dom(B)=\{  w\in\dom(A):\Phi \Gm   w=\Gp   w\}\mbox{ and }Bw=Aw\]
for all $w\in\dom(B)$.  It is straightforward that $\Phi \Gm   w=\Gp   w$ if and only if 
\[  (\Id +\Phi)\Gamma_1 w +(\Id -\Phi)\Gamma_2 w=0.\]
So we refind \cite[ Theorem 4.2]{Weg}.   

In contrast to boundary quadruples, boundary triples in the sense of Wegner do not always exist. In fact the following holds. 
\begin{prop}\label{prop:7.2}
Given a densely defined skew-symmetric operator $A_0$ on $V$, the following assertions are equivalent:
\begin{itemize}
	\item[(i)] there exists a boundary triple for $A_0$;
	\item[(ii)] $\ker (\Id -A)$ and $\ker(\Id +A)$ are isomorphic;
	\item[(iii)] $A_0$ has an extension which generates a unitary $C_0$-group;
	\item[(iv)] if $(\Hm,\Hp,\Gm  ,\Gp  )$ is a boundary quadruple, then $\Gp  W$ and $\Gm   W$ are isomorphic. 
\end{itemize}
\end{prop} 
\begin{proof}
$(iv)\Rightarrow (i)$ By assumption there exists a unitary operator $U:G_1 W\to \Gm  W$. Let $K=\Gm  W$, 	$\Gamma_{\mbox{\tiny{-}}}=\Gm  $, $\Gamma_+=U\Gp  $. Then 
$(\Gamma_{\mbox{\tiny{-}}},\Gamma_+):W\to K\times K$ is surjective and 
\begin{eqnarray*}
	\langle Au,v\rangle_W +\langle u,Av\rangle_W & = & \langle \Gp   u,\Gp   v\rangle_H-\langle \Gm   u,\Gm   v\rangle_H\\
	 & = & \langle \Gamma_+ u,\Gamma_+   v\rangle_K-\langle \Gamma   u,\Gamma  v\rangle_H.
\end{eqnarray*}
Now using Lemma~\ref{lem:7.1}
 we obtain $\Gamma_1,\Gamma_2:W\to K\times K$ such that \eqref{eq:7.6} and \eqref{eq:7.7} holds.\\
 $(i)\Rightarrow (iv)$ This follows from the lines following \eqref{eq:7.7}. \\
 $(ii)\Longleftrightarrow (iii)\Longleftrightarrow(iv)$. This is Corollary~\ref{cor:4.3}.   
\end{proof}

 \noindent \textbf{Acknowledgments:}  This research is partly supported by the B\'ezout Labex, funded by ANR, reference ANR-10-LABX-58.   The authors are grateful to N. Skrepek for an interesting discussion on the wave equation.  
 
 \bibliographystyle{abbrv}
 \bibliography{horsdoeuvrefinal-total}
\end{document}